\documentclass[11pt,letterpaper]{article}
\usepackage[left=1.3in,right=1.3in,top=1.3in,bottom=1.3in]{geometry}

\usepackage[linesnumbered,ruled]{algorithm2e}
\usepackage{verbatim}
\usepackage{mathtools}
\usepackage{indentfirst}
\usepackage{subcaption}
\usepackage{caption}
\usepackage{afterpage}
\usepackage{multirow}
\usepackage{url}
\usepackage[onehalfspacing]{setspace}
\usepackage{rotating}
\usepackage{amsthm}

\SetKwComment{Comment}{$\triangleright$\ }{}
\allowdisplaybreaks
\newtheorem{theorem}{Theorem}%

\begin{document}

\title{A reformulation-enumeration MINLP algorithm for gas network design}

\author{Yijiang Li\thanks{School of Industrial and Systems Engineering, Georgia Institute of Technology (yijiangli@gatech.edu)}, Santanu S. Dey\thanks{School of Industrial and Systems Engineering, Georgia Institute of Technology (santanu.dey@isye.gatech.edu)} and Nikolaos V. Sahinidis\thanks{School of Industrial and Systems Engineering and School of Chemical and Biomolecular Engineering, Georgia Institute of Technology (nikos@gatech.edu)}}
\date{}

\maketitle

\abstract{Gas networks are used to transport natural gas, which is an important resource for both residential and industrial customers throughout the world. The gas network design problem is generally modelled as a nonconvex mixed-integer nonlinear integer programming problem (MINLP). The challenges of solving the resulting MINLP arise due to the nonlinearity and nonconvexity. In this paper, we propose a framework to study the ``design variant'' of the problem in which the variables are the diameter choices of the pipes, the flows, the potentials, and the states of various network components. We utilize a nested loop that includes a two-stage procedure that involves a convex reformulation of the original problem in the inner loop and an efficient enumeration scheme in the outer loop. We conduct experiments on benchmark networks to validate and analyze the performance of our framework.}

\section{Introduction} \label{sec:intro}

Natural gas is a very important and common resource for both residential and industrial customers around the world. In the United States alone, a total of $27.7$ trillion cubic feet of natural gas were delivered to $77.3$ million customers in 2020 (\cite{eia_2020_demand}). To transport natural gas to meet this demand, a natural gas transportation system has been developed which was worth \$$187.9$ billion in 2020 (\cite{asset_2020}). A gas transportation system is usually modeled as a directed graph where the nodes can be customers with demands, manufacturers with supplies, or in-nodes that do not have either demands or supplies, while the arcs represent various system components. Modeling and optimizing gas transportation systems is very challenging due to the complex nature of the physical principles governing the operations of the system components. Generally, these models involve nonlinear and nonconvex constraints. Even simple models of the system components lead to challenging problems, as the scale of realistic instances is quite large compared to what state-of-the-art solvers can tackle.

In general, the system components (the arcs) include pipes, short pipes, resistors, compressors, valves, and control valves. There are additional components such as filters and measuring devices. We omit these additional components and consider the most common ones. Pipes constitute the majority of the system components. Control valves are sometimes referred to as regulators as well. Each of the system components serves a different role. The components can be grouped into passive components and active components. Pipes, short pipes, and resistors are passive system components and do not have on and off states. Compressors, control valves, and valves are active system components with on and off states. There are several types of gas network optimization problems. Most problems involve the decision on the flowrates in the arcs and the pressures at the nodes. Alternatively, it is sometimes convenient to consider node potentials, which are defined as the squared pressures. Commonly used benchmarking instances are based on the Belgian network (\cite{dewolf1996} and~\cite{dewolf2000}) of size up to 23 nodes and the GasLib networks of various sizes up to 4197 nodes (\cite{schmidt2015}). 

 In this paper, we study the {\em design problem}, which considers a given set of pipe locations. The main decisions involve choosing the pipe diameters, the states for valves, compressors, and control valves, and flowrates and potentials to transport gas to satisfy the given demand and supply scenarios while minimizing the network construction costs. We call this version of design problem the {\em design-from-scratch variant}, different from the reinforcement version. We do not include any operational costs, such as the operational cost of the compressors. In the design-from-scratch variant, we assume there are no existing pipes in the network. More details are provided in Section~\ref{sec:literature}. The demand and supply scenarios are commonly referred to as {\em nominations}. 
 
 The main contributions of this paper are the following. We propose a decomposition framework that involves an iterative procedure of solving a convex integer master problem and a verification subproblem for the solutions obtained from the master problem, and a binary search to minimize the construction cost of the pipes. We use the GasLib networks to validate our framework. Our framework is able to effectively solve the design problem on GasLib-582 network, which contains 582 nodes. To the best of our knowledge, this is the first paper that solves the gas network design problem on such large-scale instances. Previous literature (reviewed below in Section~\ref{sec:literature}) on the design-from-scratch version of the problem has not considered any instance with over 500 nodes, and these works do not simultaneously consider active elements, discrete diameter choices, and general (non-tree) underlying networks.
 

The structure of the remainder of this paper is as follows. We give a summary of the relevant literature in Section~\ref{sec:literature}. Section~\ref{sec:origin_formulation} presents the technical background and a compact formulation for the design problem while Section~\ref{sec:decomposition_framework} presents the decomposition framework. Implementation considerations and numerical experiments to validate the decomposition framework are presented in Section~\ref{sec:numerical_experiments}. Lastly, in Section~\ref{sec:conclusion}, we present concluding remarks and future research directions.

\section{Literature review} \label{sec:literature}
Gas network systems have been an important topic of study in the past several decades. As the relevant literature is rather extensive, here we review only works that are most closely related to ours. For a detailed review of the literature, we refer the interested readers to~\cite{rio-Mercado2015} and~\cite{zheng2010}. In addition,~\cite{Hante_2023} provides an overview on the modeling and common solution approaches in gas network systems. 

Among the types of problems studied, we focus below on two relevant problem types, the nomination validation problem and design problem. In the {\em nomination validation problem}, we assume a nomination is given. In the case where active system components are not considered, the problem aims to evaluate whether the existing network topology is feasible with respect to the given nomination. In the case where active system components are involved, the problem aims to determine whether there exist feasible configurations for the active system components along with rest of the components such that the resulting network is feasible with respect to the nomination. Work in this area includes~\cite{liers2020},~\cite{pfetsch2015},~\cite{geibler2015},~\cite{geibler2017},~\cite{dewolf2000},~\cite{burlacu2017},~\cite{koch2015},~\cite{rose2016}, and~\cite{schmidt2016}. In particular,~\cite{liers2020} presents a convexification scheme to find the convex hull of a ``Y'' junction in the network to deal with the nonconvex constraints arising from the governing physical principles. The paper~\cite{pfetsch2015} presents four approaches for solving nomination validation problems. The first approach is a piecewise linear approximation scheme that utilizes the generalized incremental model to linearize the nonlinear constraints; the approximation is improved iteratively by adding more linearization points. The second approach is a spatial branch-and-bound algorithm, which iteratively partitions the feasible region and refines the estimations and relaxations of the original problem in each partition to obtain dual bounds on the solution. The next approach in~\cite{pfetsch2015} is called RedNLP, a two-stage procedure, in which heuristics and reformulations are employed to find promising configurations of the active system components and the feasibility of configurations is checked in the second stage. The last approach considered is called the smoothing procedure, commonly used for mathematical programs with equilibrium constraints, to model the discrete decisions corresponding to the configurations of the active system components with continuous variables. Numerical experiments and comparisons across the four approaches are performed on the GasLib-582 network. Overall, the spatial branch-and-bound outperforms the other three approaches. The papers by~\cite{geibler2015} and~\cite{geibler2017} consider additional constraints of satisfying heat-power demand and supply in the nomination validation problem. In both works, an alternating direction method is applied to a linearized approximation model and numerical experiments are performed on the GasLib-4197 network. 

The design problem can be divided into the reinforcement problem and the design-from-scratch problem. In the {\em reinforcement problem}, it is assumed that an existing topology is given. The problem considers the options to install additional system components, mostly pipes and compressors, to satisfy a given nomination while minimizing the construction costs of the new system components. The cost of a new pipe is usually a function of its diameter and, as a result, diameter selection becomes a decision for the pipes. The {\em design-from-scratch problem} assumes no existing pipes in the network and makes decisions on the diameters for all pipes. In both the reinforcement problem and the design-from-scratch problem, the diameter choices can be continuous or discrete. Recent works that study the reinforcement problem include  ~\cite{babonneau2012},~\cite{borraz-Sanchez2015},~\cite{zhang1996},~\cite{dewolf1996},~\cite{shiono2016}, and~\cite{robinius2019}. In particular,~\cite{babonneau2012} considers the reinforcement problem with continuous diameters and utilizes a two-stage formulation. The first-stage problem is a convex nonlinear program to compute favorable diameter choices and flowrates, while the second-stage problem checks whether the first-stage solution is feasible with respect to the nomination by solving for the potential at each node. This convex program was formally introduced in~\cite{cherry1951} and adapted to solve a network problem in~\cite{collins197778}. The numerical experiments were performed on the Belgian network. The work in~\cite{borraz-Sanchez2015} considers also a reinforcement problem with discrete diameters. To deal with the nonconvexity from the governing physical principles, the paper considers reformulations and a convex relaxation which is second-order cone representable. These authors also utilize perspective strengthening that is studied in~\cite{frangioni2006} and~\cite{oktay2009} to enhance the relaxation. Numerical experiments were performed on both Belgian and GasLib networks. For the design-from-scratch problem,~\cite{zhang1996} considers discrete diameter choices without any active components. A bilevel formulation is proposed and in solving the formulation, the discrete variables corresponding to the discrete diameter choices are first transformed to continuous variables. Subsequently, the lower-level problem is reformulated via conjugate duality while a trust region algorithm is developed for the upper-level problem. Numerical experiments were performed using two networks of size up to 14 nodes. The paper by~\cite{dewolf1996} studies a variant of the design-from-scratch problem with continuous diameter choices and solve the model by a bundle method with a generalized gradient. Numerical experiments are performed with the Belgian network. Both~\cite{shiono2016} and~\cite{robinius2019} consider the problem on a tree-shaped network. \cite{robinius2019} consider continuous diameter and ``approximate discrete diameter'' obtained from the optimal continuous diameter. These authors develop an iterative procedure to solve the problem; their approach contracts the tree (network) converting the original tree into a single equivalent arc. Numerical illustrations of this procedure were performed on networks of size up to 36 nodes. In addition to solving the design problem,~\cite{robinius2019} proposes a tractable framework to consider infinitely many demand scenarios such that the diameter choices are feasible for all demand scenarios on a tree-shaped network.

As water flow is governed by similar physical equations as gas flow does due to their fluid nature, the water network design problem shares common characteristics as the gas network design problem. We refer interested readers to an overview of water network design problem in~\cite{dambrosio2015} and some modeling and solution techniques in~\cite{bragalli2008},~\cite{bragalli2012}, and~\cite{10.1007/11841036_62}. 

\section{Problem description} \label{sec:origin_formulation}
\subsection{Technical background}
In this section, we provide necessary technical background on gas networks that we will need for our formulations later. For more details, we refer the interested readers to~\cite{humpola2014},~\cite{pfetsch2015},~\cite{schmidt2015}, and~\cite{Hante_2023}. For the remainder of this paper, we use a directed graph $G = (\mathcal{V}, \mathcal{A})$ to represent a gas network, where each arc $a \in \mathcal{A}$ represents a system component of the network and each node $v \in \mathcal{V}$ can be a customer, a manufacturer, or an in-node. For each node $v \in \mathcal{V}$ we track its pressure $p_v$ and potential $\pi_v$, where the pressure and potential are related by the equation:
\begin{eqnarray}\label{eq:pressurepotential}
\pi_v = p_v^2.
\end{eqnarray}
We denote lower and upper bounds on the potential at a node $v$ by $\pi_v^{\text{min}}$ and $\pi_v^{\text{max}}$, respectively.

\subsubsection{Pipes}
As mentioned earlier, the majority of the arcs in gas networks are pipes. A pipe $a = (v,w)$ is specified by its length $l$, diameter $D$, and material properties. The flowrate in arc $a$, denoted as $q_a$, is upper bounded by a value $q_a^\text{max}$ which is determined by the cross-sectional area, $A := \pi D^2/4$, and material properties. We assume a linear relation between the value of $q_a^\text{max}$ and the cross-sectional area, i.e.,
\begin{equation}
q_a^\text{max} \propto A \label{q_max_A}.	
\end{equation}

The gas flow in pipe $a = (v,w)$ is described by a set of partial differential equations derived from conservation of mass and conservation of momentum which, under certain assumptions, can be simplified to  
\begin{equation}
   \pi_v - \pi_w =  p_v^2 - p_w^2 = \alpha_a \vert q_a \vert q_a, \label{pipe_physics}
\end{equation}
where $\alpha_a$ is the pressure loss coefficient. The pressure loss coefficient, $\alpha_a$, depends on the material and diameter of the pipe and a few properties of natural gas. 
As pipes allow bi-directional flow, the sign of the potential drop depends on the direction of the flow resulting in the absolute value of the flowrate variable $q_a$ in the equation. 

\subsubsection{Short pipes}
Short pipes are used for modeling purposes to handle complicated contract situations and are modeled as lossless pipes, i.e., a short pipe $a$ is a regular pipe with $\alpha_a = 0$.

\subsubsection{Resistors}
Resistors are commonly used to model pressure or potential drop. In this work, we assume resistors behave in the same way as pipes in terms of potential drop. We refer to~\cite{pfetsch2015} for alternative ways to model resistors.

\subsubsection{Compressors}
Compressors are used to increase the potential along an arc. There are many models proposed for compressors. In this paper, we adopt the model used in~\cite{babonneau2012}. For a compressor $a = (v,w)$, we use a binary variable $z_a$ to indicate the on and off states where $z_a = 1$ indicates the compressor is on and $z_a = 0$ otherwise. When the compressor is on, it allows flow from $v$ to $w$ and increases the potential from $v$ to $w$. When it is off, it does not allow any flow. As a result, we have the following relations for a compressor:
\begin{align}
& \pi_v - \pi_w \le 0, \; q_a \ge 0,\; \text{if} \; z_a = 1,\\
& q_a = 0, \; \text{if} \; z_a = 0.	
\end{align}
Additionally, there are limits on potential ratio as follows:
\begin{equation}
    \kappa_a^\text{min} \pi_v \le \pi_w \le \kappa_a^\text{max} \pi_v,
\end{equation}
where $\kappa_a^\text{min} = 1$ and $\kappa_a^\text{max} \ge 1$ are typical for a compressor (see~\cite{borraz-Sanchez2015}). Furthermore, when a compressor $a = (v,w)$ is on, it can impose additional bounds on the potentials at nodes $v$ and $w$. We denote those bounds by $(\pi_v^\text{min})^\prime$ and $(\pi_w^{\text{max}})^\prime$.

\subsubsection{Valves}
Valves are incorporated in the network to join or separate two nodes. Valves have the states of being open and closed. They allow bi-directional flow when they are open. A binary variable $z_a$ is used to model the states of valves. For a valve $a = (v,w)$, $z_a = 1$ indicates the valve is open and $z_a = 0$ otherwise. When a valve is open, the potentials at the two end nodes have to be equal. When a valve is closed, it does not allow any flow. Formally, the constraints for a valve $a = (v,w)$ are expressed as follows:
\begin{align}
	& \pi_v = \pi_w, \; q_a \text{ arbitrary}, \; \text{if} \; z_a = 1,\\
	& q_a = 0, \; \pi_v, \pi_w \text{ arbitrary}, \; \text{if} \; z_a = 0.
\end{align}

\subsubsection{Control valves}
Contrary to a compressor, the presence of a control valve in the network results in potential relief. We adopt a similar model that is used for compressors. For a control valve $a = (v,w)$, a binary variable $z_a$ is used to indicate its states. When it is on, it allows flow from $v$ to $w$ and causes a potential relief from $v$ to $w$. When it is off, it does not allow any flow. We have the following model for the control valve:
\begin{align}
& \pi_v - \pi_w \ge 0,\; q_a \ge 0, \; \text{if} \; z_a = 1,\\
& q_a = 0, \; \text{if} \; z_a = 0.
\end{align}
The limits on the potential relief are given by
\begin{equation}
	\kappa_a^\text{min} \pi_v \le \pi_w \le \kappa_a^\text{max} \pi_v,
\end{equation}
where $\kappa_a^\text{min} > 0$ and $\kappa_a^\text{max} \le 1$ are typical for a control valve (see~\cite{borraz-Sanchez2015}). A control valve $(v,w)$ can impose additional bounds on the potentials at nodes $v$ and $w$ when it is on. Similar to compressors, we denote those bounds by $(\pi_v^\text{min})^\prime$ and $(\pi_w^{\text{max}})^\prime$.

\subsection{Design problem}
We consider the design of a gas network for a given set of pipe locations (arcs), whose diameters we must decide. As discussed in the previous section, we have different system components and thus we divide the arc set $\mathcal{A}$ into $\mathcal{A} = A_p \cup A_{sp} \cup A_r \cup A_{cp} \cup A_v \cup A_{cv}$ where $A_p$, $A_{sp}$, $A_r$, $A_{cp}$, $A_v$, and $A_{cv}$ are the set of pipes, short pipes, resistors, compressors, valves, and control valves, respectively. We consider discrete diameter choices in our setting and denote the diameter choices by the set $[n] := \{1,2,\ldots, n\}$. We use binary variables $z_{a,i}$, $a \in A_p$ and $i \in [n]$, to model the discrete diameter choices of the pipes. We further denote the length and diameter of the pipe $a \in A_p$ with the diameter choice $i \in [n]$ by $l_a$ and $D_{a,i}$, respectively, and we use the same cost function, $f_{a,i}$, that is used in~\cite{babonneau2012} and~\cite{borraz-Sanchez2015}, namely
\begin{equation}
    f_{a,i} = l_a(1.04081^{-6}D_{a,i}^{2.5} + 11.2155).
\end{equation}

The cost function, $f_{a,i}$, can be computed before the execution of the design problem and thus considered given. Additionally, there is a trade-off in the selection of the diameters. A larger diameter, on one hand, leads to a smaller potential drop coefficient and a higher maximum flowrate, while, on the other hand, leads to a larger cost $f_{a,i}$. We introduce a flow direction variable $x_a^\text{dir} \in \{0,1\}$ for $a \in A_p \cup A_{sp} \cup A_r \cup A_v$ to account for the bidirectional flow. Recall that compressors and control valves only allow one flow direction. For a pipe $a \in A_p$, as a result of the multiple diameter choices, we have multiple flowrate variables $q_{a,i}$ for each $i \in [n]$. We decompose the flow into positive flow and negative flow, i.e., for $a \in A_p$, $q_{a,i} = q_{a,i}^+ - q_{a,i}^-$ and for $a \in A_{sp} \cup A_r \cup A_v$, $q_a = q_a^+ - q_a^-$. The maximum flowrate limit $q_a^\text{max}$ can be defined individually for each diameter choice $i$ as $q_{a,i}^\text{max}$ by relation (\ref{q_max_A}). Similarly, the potential drop coefficients $\alpha_{a,i}$ can be computed for each diameter choice. Furthermore, for a node $v \in \mathcal{V}$, we denote the set of incoming arcs and outgoing arcs by $A_\text{in}(v)$ and $A_\text{out}(v)$, respectively, i.e., $A_\text{in}(v) = \{a \in \mathcal{A} \vert a=(w,v)\}$ and $A_\text{out}(v) = \{a\in \mathcal{A} \vert a=(v,w)\}$. We use $d_v$ to denote the demand or supply at a node $v \in \mathcal{V}$.

With this notation and technical background, we give a MINLP formulation to the design problem:
{\small
\begin{align}
	\min_{z, q^+, q^-, \pi, x^\text{dir}} \; & \sum_{a \in A_p}\sum_{i \in [n]} f_{a,i} z_{a,i}, \label{compact_obj}\\
	\text{s.t.}\; & \sum_{a \in A_\text{in}(v) \backslash A_p} q_a^+ - \sum_{a \in A_\text{in}(v) \backslash A_p \cup A_{cp} \cup A_{cv}} q_a^- \notag \\
    & - \left(\sum_{a \in A_\text{out}(v) \backslash A_p} q_a^+ - \sum_{a \in A_\text{out}(v) \backslash A_p \cup A_{cp} \cup A_{cv}} q_a^- \right) \notag \\ 
	& + \sum_{i \in [n]} \sum_{a \in A_\text{in}(v) \cap A_p} (q_{a,i}^+ - q_{a,i}^-) - \sum_{i \in [n]}\sum_{a \in A_\text{out}(v) \cap A_p} (q_{a,i}^+ - q_{a,i}^-) = d_v, \qquad v \in \mathcal{V}, \label{compact_constr_flow_conservation}\\
	& \pi_v^{\text{min}} \le \pi_v \le \pi_v^{\text{max}}, \qquad v \in \mathcal{V}, \label{compact_constr_potential_bound}\\
	& x_a^{\text{dir}} \in \{0,1\}, \qquad a \in A_p \cup A_{sp} \cup A_v \cup A_r, \label{compact_constr_direction}\\
	& 0 \le q_{a,i}^-,\; q_{a,i}^+ \le q_{a,i}^{\text{max}}z_{a,i}, \qquad i \in [n], a \in A_p,\label{compact_constr_pipe_1}\\
	& \pi_v - \pi_w = \sum_{i \in [n]} \alpha_{a,i} (q_{a,i}^+)^2 - \sum_{i \in [n]} \alpha_{a,i} (q_{a,i}^-)^2, \qquad a \in A_p, \label{compact_constr_pipe_2}\\
	& 0 \le q_{a,i}^+ \le q_{a,i}^{\text{max}} x_a^{\text{dir}}, \qquad a \in A_p, i \in [n],\label{compact_constr_pipe_3}\\
	& 0 \le q_{a,i}^- \le q_{a,i}^{\text{max}} (1 - x_a^{\text{dir}}), \qquad a \in A_p, i \in [n],\label{compact_constr_pipe_4}\\
	& \sum_{i \in [n]} z_{a,i} = 1, \qquad a \in A_p, \label{compact_constr_pipe_5}\\
	& \pi_v = \pi_w, \qquad a \in A_{sp}, \label{compact_constr_shortpipe_1}\\
	& 0 \le q_a^+ \le q_a^{\text{max}} x_a^{\text{dir}}, \qquad a \in A_{sp}, \label{compact_constr_shortpipe_2}\\
    & 0 \le q_a^- \le q_a^{\text{max}} (1 - x_a^{\text{dir}}), \qquad a \in A_{sp},\label{compact_constr_shortpipe_3}\\
	& \pi_v - \pi_w = \alpha_a (q_a^+)^2 - \alpha_a (q_a^-)^2, \qquad a \in A_r, \label{compact_constr_resistor_1}\\
	& 0 \le q_a^+ \le q_a^{\text{max}} x_a^{\text{dir}}, \qquad a \in A_r, \label{compact_constr_resistor_2}\\
	& 0 \le q_a^- \le q_a^{\text{max}} (1 - x_a^{\text{dir}}), \qquad a \in A_r, \label{compact_constr_resistor_3}\\
 	& \kappa_a^\text{min} \pi_v - M(1-z_a) \le \pi_w \le \kappa_a^\text{max} \pi_v + M(1-z_a), \qquad a \in A_{cp} \cup A_{cv},\label{compact_constr_comp_control_1}\\
 	& 0 \le q_a^+ \le q_a^{\text{max}} z_a, \qquad a \in A_{cp} \cup A_{cv},\label{compact_constr_comp_control_2}\\
 	& (\pi_v^\text{min})^\prime - M(1-z_a) \le \pi_v, \qquad a = (v,w) \in A_{cp} \cup A_{cv}, \label{compact_constr_comp_control_3}\\
 	& \pi_w \le (\pi_w^{\text{max}})^\prime + M(1-z_a), \qquad a = (v, w) \in A_{cp} \cup A_{cv}, \label{compact_constr_comp_control_4}\\
 	& \pi_v - \pi_w \le M(1-z_a), \qquad a \in A_v, \label{compact_constr_valve_1}\\
 	& \pi_v - \pi_w \ge -M(1-z_a), \qquad a \in A_v, \label{compact_constr_valve_2}\\
 	& 0 \le q_a^-, q_a^+ \le q_a^{\text{max}}z_a, \qquad a \in A_v, \label{compact_constr_valve_3}\\
 	& 0 \le q_a^+ \le q_a^{\text{max}} x_a^{\text{dir}}, \qquad a \in A_v, \label{compact_constr_valve_4}\\
	& 0 \le q_a^- \le q_a^{\text{max}} (1 - x_a^{\text{dir}}), \qquad a \in A_v.\label{compact_constr_valve_5}
\end{align}
}

Note that we omit the subscripts for the variables under the ``min'' and variables $z$ could either be the diameter choices of the pipes or the states of the compressors, the valves, and the control valves. In this model, the objective function (\ref{compact_obj}) minimizes the construction cost of the pipes, also known as the budget. The scalar $M$ represents a large number. For the constraints that involve $M$, we can alternatively write them in a nonlinear fashion, which eliminates the need for big-$M$s. In particular, for constraints (\ref{compact_constr_comp_control_1}) and (\ref{compact_constr_comp_control_3})-(\ref{compact_constr_comp_control_4}), we have
\begin{align}
	& z_a \kappa_a^\text{min} \pi_v \le \pi_w, \qquad a \in A_{cp} \cup A_{cv}, \label{compact_constr_comp_control_p_1}\\
	& z_a \pi_w \le \kappa_a^\text{max} \pi_v, \qquad a \in A_{cp} \cup A_{cv}, \label{compact_constr_comp_control_p_2}\\
	& z_a (\pi_v^\text{min})^\prime \le \pi_v, \qquad a \in A_{cp} \cup A_{cv}, \label{compact_constr_comp_control_p_3}\\ 
	& z_a \pi_w \le (\pi_w^\text{max})^\prime, \qquad a \in A_{cp} \cup A_{cv}, \label{compact_constr_comp_control_p_4}, 
\end{align}
and for constraints (\ref{compact_constr_valve_1})-(\ref{compact_constr_valve_2}), we have
\begin{align}
    & (\pi_v - \pi_w)z_a = 0, \qquad a \in A_v \label{compact_constr_valve_p_1}.
\end{align}

We group the constraints in blocks with block names and provide a summary in Table~\ref{table:compact_constraints}. We will refer to the corresponding set of constraints by their block name. 
\begin{table}[!htbp]
\centering
\begin{tabular}{|c|c|l|}
	\hline
	Constraints & Block names & Explanations \\\hline\hline
	(\ref{compact_constr_flow_conservation}) & \texttt{Flow\_conserv} & Flow conservation\\\hline
	(\ref{compact_constr_potential_bound})-(\ref{compact_constr_direction}) & \texttt{Bound} & Bounds on potentials (\ref{compact_constr_potential_bound}); binary directions (\ref{compact_constr_direction})\\\hline
	\multirow{2}{*}{(\ref{compact_constr_pipe_1})-(\ref{compact_constr_pipe_5})} & \multirow{2}{*}{\texttt{Pipe}} & Flow limits on diameter choices (\ref{compact_constr_pipe_1}); potential drop (\ref{compact_constr_pipe_2});\\
	& & flow limits on directions (\ref{compact_constr_pipe_3})-(\ref{compact_constr_pipe_4}); diameter selection (\ref{compact_constr_pipe_5})\\\hline
	(\ref{compact_constr_shortpipe_1})-(\ref{compact_constr_shortpipe_3}) & \texttt{Short\_pipe} & Potential (\ref{compact_constr_shortpipe_1}); flow limits on directions (\ref{compact_constr_shortpipe_2})-(\ref{compact_constr_shortpipe_3})\\\hline
	(\ref{compact_constr_resistor_1})-(\ref{compact_constr_resistor_3}) & \texttt{Resistor} & Potential drop (\ref{compact_constr_resistor_1}); flow limit on directions (\ref{compact_constr_resistor_2})-(\ref{compact_constr_resistor_3})\\\hline
	\multirow{2}{*}{(\ref{compact_constr_comp_control_1})-(\ref{compact_constr_comp_control_4})} & \texttt{Comp\_and\_} & Depend on on/off states; Potential increase/relief limit (\ref{compact_constr_comp_control_1});\\
	& \texttt{cont\_valve} & flow limit (\ref{compact_constr_comp_control_2}); additional bounds (\ref{compact_constr_comp_control_3})-(\ref{compact_constr_comp_control_4})\\\hline
	\multirow{2}{*}{(\ref{compact_constr_valve_1})-(\ref{compact_constr_valve_5})} & \multirow{2}{*}{\texttt{Valve}} & Depend on on/off states; Potential (\ref{compact_constr_valve_1})-(\ref{compact_constr_valve_2});\\
	& & flow limit (\ref{compact_constr_valve_3}); flow limits on directions (\ref{compact_constr_valve_4})-(\ref{compact_constr_valve_5})\\\hline
	\multirow{2}{*}{(\ref{compact_constr_comp_control_2}), (\ref{compact_constr_comp_control_p_1})-(\ref{compact_constr_comp_control_p_4})} & \texttt{Comp\_and\_} & \multirow{2}{*}{Same as \texttt{Comp\_and\_cont\_valve block}}\\
	&\texttt{cont\_valve\_nl} & \\\hline
	\multirow{2}{*}{(\ref{compact_constr_valve_3})-(\ref{compact_constr_valve_5}), (\ref{compact_constr_valve_p_1})} & \multirow{2}{*}{\texttt{Valve\_nl}} & \multirow{2}{*}{Same as \texttt{Valve} block}\\
	& &\\\hline	
\end{tabular}
\caption{Constraint blocks} \label{table:compact_constraints}
\end{table}

Note that the above formulation can be extended to the reinforcement problem by considering, for each existing pipe, an additional diameter choice 
with no cost along with potential loss equation (\ref{pipe_physics}) in which the potential loss coefficient is computed based on the diameter of the existing pipe. 

\section{Decomposition framework} \label{sec:decomposition_framework}
We now present a decomposition framework to solve the design problem. The decomposition consists of three major components: primal bound loop, binary search on budget, and initial budget search. Before we present the details on each component, we introduce more background on the convex program introduced in~\cite{cherry1951} and adapted in~\cite{collins197778}, which was mentioned briefly in literature review. 
\subsection{CVXNLP} \label{subsec:CVXNLP}
The convex program is called (CVXNLP) in~\cite{raghunathan2013}; we adopt the same name. We base the discussions of (CVXNLP) on a gas network in contrast to a water network in~\cite{raghunathan2013} in this section for completeness. For a network with only pipes, i.e., $\mathcal{A} = A_p$, (CVXNLP) is closely related to the following set of \textit{network analysis equations}:
\begin{align}
& \pi_v - \pi_w = \text{sgn}(q_a) \phi(\vert q_a \vert ), \qquad a \in A_p, \label{network_ana_potential_drop} \\
& \sum_{a \in A_\text{in}(v)} q_a - \sum_{a \in A_\text{out}(v)} q_a = d_v, \qquad v \in \mathcal{V} \label{network_ana_flow_conserv},
\end{align}
where $\text{sgn}(\cdot)$ is the sign function and $\phi(\cdot)$ is the potential loss function. In the network analysis equations, (\ref{network_ana_potential_drop}) is the potential drop equation and (\ref{network_ana_flow_conserv}) is the flow conservation. (CVXNLP) is formally given by
\begin{align}
\min_{q^+, q^-} \; & \sum_{a \in A_p} \Phi(q_a^+) + \Phi(q_a^-), \\
\text{s.t.} \; & \sum_{a \in A_\text{in}(v)} (q_a^+ - q_a^-) - \sum_{a \in A_\text{out}(v)} (q_a^+ - q_a^-) = d_v, \qquad v \in \mathcal{V}, \label{CVXNLP_flow_conserv}\\
& 0 \le q_a^-, q_a^+, \qquad a \in A_p, \label{CVXNLP_bounds}
\end{align}
where $\Phi(\cdot)$ is defined by
\begin{equation}
	\Phi(q) = \int_0^q \phi(q^\prime)dq^\prime.
\end{equation}
(CVXNLP) is formally linked to the network analysis equations by the following theorem.

\begin{theorem} \label{CVXNLP_network_eqv}
	If the potential loss function $\phi(\cdot)$ is strictly monotonically increasing function of flowrate, $q$, with $\phi(0) = 0$, then there exists a solution $(\pi, q)$ to the network analysis equations if and only if there exists a solution $(\hat{q}^+, \hat{q}^-, \hat{\lambda}, \hat{\mu}^+, \hat{\mu}^-)$ to (CVXNLP) where $\lambda$, $\mu^+$, and $\mu^-$ are dual variables to the flow conservation constraint (\ref{CVXNLP_flow_conserv}) and bounds constraints (\ref{CVXNLP_bounds}), respectively. 
\end{theorem}
\begin{proof}
The proof is adapted from a proof in~\cite{raghunathan2013} and can be found in Appendix~\ref{sec:appendix_a}. 
\end{proof}
The monotonicity assumption needed for the theorem to hold is commonly satisfied by gas networks. In addition, as a result of the equivalence between (CVXNLP) and the network analysis equations stated in  Theorem~\ref{CVXNLP_network_eqv}, we can solve the convex (CVXNLP) in lieu of the nonconvex network analysis equations and obtain a solution $(\pi, q)$. As there are no bounds enforced in the network analysis equations for $\pi$, we have to perform an additional step to verify that $\pi$ satisfies the corresponding bounds.

\subsection{Primal bound loop} \label{subsec:primal_bound_loop}
The equivalence discussed in Section~\ref{subsec:CVXNLP} motivates us to develop a decomposition procedure that solves a variant of (CVXNLP) as a master problem, while a subproblem is used to check for feasibility. We call this procedure the primal bound loop, which checks whether a budget $C$ is feasible with respect to a nomination. The master problem, denoted by ($P_m$), is based on (CVXNLP) and is as follows:
\begin{align}
(P_m) \;\;\;\; \min_{z,q^+,q^-} \; & \sum_{i \in [n]} \sum_{a \in A_p} \frac{\alpha_{a,i}}{3}(q_{a,i}^+)^3 + \frac{\alpha_{a,i}}{3}(q_{a,i}^-)^3 + \sum_{a \in A_r} \frac{\alpha_a}{3} (q_a^+)^3 + \frac{\alpha_a}{3} (q_a^-)^3, \label{primal_master_obj}\\
\text{s.t.}\; & (\ref{compact_constr_flow_conservation}),(\ref{compact_constr_pipe_1}), (\ref{compact_constr_pipe_5}), (\ref{compact_constr_comp_control_2}), (\ref{compact_constr_valve_3}), \notag \\
& 0 \le q_a^-, q_a^+ \le q_a^\text{max}, \qquad a \in A_{sp} \cup A_r, \label{primal_master_sp_r_bound}\\
& \sum_{a \in A_p}\sum_{i \in [n]} f_{a,i} z_{a,i} \le C. \label{primal_master_budget}
\end{align}
In this model, the objective function (\ref{primal_master_obj}) extends (CVXNLP) to account for multiple flow variables $q_{a,i}^+$ for $a \in A_p$ and $i \in [n]$. Constraint (\ref{compact_constr_flow_conservation}) is the flow conservation. Constraints (\ref{compact_constr_pipe_1}) and (\ref{compact_constr_pipe_5}) on the pipes ensure one diameter choice is selected and the corresponding flow limit is enforced. Constraints (\ref{compact_constr_comp_control_2}) and (\ref{compact_constr_valve_3}) on the active system components ensure flows are only allowed when the corresponding binaries are on. Constraint (\ref{primal_master_sp_r_bound}) enforces the flow limit on the short pipes and resistors. The last constraint (\ref{primal_master_budget}) is a budget constraint on the construction cost of pipes. 

The above model differs from (CVXNLP) mainly in two ways. Firstly, we have introduced the diameter choices, $z_{a,i}$ for $a \in A_p$ and $i \in [n]$ and configurations for the active system components, $z_a$ for $a \in A_{cp} \cup A_{cv} \cup A_v$. If the diameter choices and configurations of the active system components are fixed, ($P_m$) resembles the original (CVXNLP). Secondly, we have a constraint to upper bound the construction cost of pipes by the budget $C$. We hope to obtain favorable diameter choices and active system component configurations from solving this modified (CVXNLP) due to the equivalence between (CVXNLP) and the network analysis equations shown in Theorem~\ref{CVXNLP_network_eqv}. In the solution of ($P_m$), we denote the optimal diameter choices by $z_{a, i}^*$ for $a \in A_p$ and $i \in [n]$ and the optimal active system configurations by $z_a^*$ for $a \in A_{cp} \cup A_{cv} \cup A_v$. We can then compute the potential loss coefficient and the flow limit of each pipe as follows:
\begin{align}
    \alpha_a & = \sum_{i \in [n]} \alpha_{a, i} z_{a,i}^*, \qquad a \in A_p, \label{expansion_pressure_loss}\\
    q_{a}^\text{max} & = \sum_{i \in [n]} q_{a,i}^\text{max} z_{a,i}^*, \qquad a \in A_p. \label{expansion_flow_limits} 
\end{align}
For the subproblem, denoted by ($P_s$), since we have additional active system components for which the constraints governing their corresponding potential changes are not included in network analysis equations,  we solve a variant of the nomination validation problem, instead of performing simple bound violation verifications, to check if the diameter choices and configurations of the active system components are feasible with respect to the nomination. ($P_s$) is given by:
\begin{align}
   (P_s) \;\;\;\; & \; \text{Find } \; q^+, q^-, x^\text{dir}, \pi, \label{subproblem_obj}\\
	\text{s.t.}\; & \texttt{Flow\_conserv}, \notag \\
	& \texttt{Bound}, \notag\\
	& \texttt{Pipe}, \notag\\
	& \texttt{Short\_pipe}, \notag\\
	& \texttt{Resistor}, \notag \\
	& \texttt{Comp\_and\_cont\_valve(\_nl)}, \notag\\
	& \texttt{Valve(\_nl)}. \notag
\end{align}
In this nomination validation problem, we solve a feasibility problem with seven blocks of constraints that are simplified from the blocks in Table~\ref{table:compact_constraints}. We list the changes to the constraints as follows.

\texttt{Flow\_conserv:} with the diameter choices fixed, we only need two flow variables, $q_a^+$ and $q_a^-$, for each pipe $a \in A_p$ and the simplified flow conservation constraint is given by:
\begin{equation}	
\sum_{a \in A_\text{in}(v)} q_a^+ - \sum_{a \in A_\text{in}(v) \backslash A_{cp} \cup A_{cv}} q_a^- - \left(\sum_{a \in A_\text{out}(v)} q_a^+ - \sum_{a \in A_\text{out}(v) \backslash A_{cp} \cup A_{cv}} q_a^- \right) = d_v, \qquad v \in \mathcal{V}. \label{subproblem_flow_conservation}\\
\end{equation}

\texttt{Pipe:} with the diameter choices determined, we compute the potential loss coefficients and flow limits, and we write the potential loss constraints as 
\begin{equation}
	\pi_v - \pi_w = \alpha_a (q_a^+)^2 - \alpha_a (q_a^-)^2, \qquad a \in A_p, \label{subproblem_pipe_1}
\end{equation}
and use the computed flow limits $q_a^\text{max}$ from (\ref{expansion_flow_limits}) as:
\begin{align}
	& 0 \le q_a^+ \le q_a^{\text{max}} x_a^{\text{dir}}, \qquad a \in A_p, \label{subproblem_pipe_2}\\
	& 0 \le q_a^- \le q_a^{\text{max}} (1 - x_a^{\text{dir}}), \qquad a \in A_p.\label{subproblem_pipe_3}
\end{align}

\texttt{Comp\_and\_cont\_valve(\_nl):} we fix the compressor and control valve configurations obtained in ($P_m$). The constraints are then linear and free of $M$. 

\texttt{Valve(\_nl):} we fix the valve configurations obtained in ($P_m$). The constraints are then linear and free of $M$. 

There are no changes to other constraints. 

Note that this nomination validation problem is still nonconvex due to constraint (\ref{subproblem_pipe_1}) and the potential loss constraint for resistors. There can be two outcomes from solving this subproblem ($P_s$). If it is infeasible, we can add an integer no-good cut  to the master problem ($P_m$) of the form:
\begin{align}
    & \sum_{a \in A_{cp} \cup A_{cv} \cup A_v, z_a^* = 0} z_a + \sum_{a \in A_{cp} \cup A_{cv} \cup A_v, z_a^* = 1} (1- z_a) \notag \\ 
    & + \sum_{i \in [n]}\sum_{a \in A_p, z_{a,i}^* = 0} z_{a,i} +  \sum_{i \in [n]}\sum_{a \in A_p, z_{a,i}^* = 1} (1 - z_{a,i}) \ge 1. \label{subproblem_nogoodcut}
\end{align}
If it is feasible, we call budget $C$ a feasible budget with respect to the nomination. 

The iterative procedure terminates when we obtain a feasible budget or when the master problem ($P_m$) becomes infeasible after adding some integer no-good cuts. In the latter case, we call budget $C$ an infeasible budget.

\subsection{Binary search on budget}
As the goal of this design problem is to minimize the construction cost of the network, we propose a binary search procedure to do so. A feasible budget from the primal bound loop provides an upper bound, $\overline{C}$, on the budget while an infeasible budget provides a lower bound, $\underline{C}$, on the budget. We present the binary search in Algorithm~\ref{binary_search_C}. For the termination conditions of the binary search in Line~\ref{binary_termination}, we consider a time limit, absolute gap $\varepsilon_e$, i.e, $\overline{C} - \underline{C} < \varepsilon_e$, or relative gap $\varepsilon_r$, i.e., $(\overline{C} - \underline{C}) / \underline{C} < \varepsilon_r$. 

\begin{algorithm}[!htbp]
\caption{Binary search of budget} \label{binary_search_C}
\textbf{Input:} Initial budget $C$, upper bound $\overline{C} = \infty$ and valid lower bound $\underline{C}$ by budget initialization \Comment*[h]{see Section~\ref{budget_initialization}}\\
\While{not terminated \label{binary_termination}}{
Solve problems ($P_m$) and ($P_s$) with budget $C$ \Comment*[h]{primal bound loop}\\
\uIf{$C$ is a feasible budget}{
	\lIf{$C < \overline{C}$}{$\overline{C} = C$}
	\lIf{$C / 2 \le \underline{C}$}{$C = (\underline{C} + C) / 2$}
	\lElse{$C = C / 2$}
}
\uElseIf{$C$ is an infeasible budget}{
	\lIf{$C > \underline{C}$}{$\underline{C} = C$}
	\lIf{$2C \ge \overline{C}$}{$C = (\overline{C} + \underline{C}) / 2$}
	\lElse{$C = 2C$}
}
\Else(\Comment*[h]{primal bound loop terminated due to time limit}){
	\lIf{$2C \ge \overline{C}$}{$C = (\overline{C} + C) / 2$}
	\lElse{$C = 2C$}
}
} 
\end{algorithm}

\subsection{Initial budget search} \label{budget_initialization}
To obtain a better initial starting budget for the binary search, we propose the following initial budget search procedure. This procedure is again an iterative procedure involving a master problem and a subproblem. The master problem, denoted by ($I_m$), is given by
\begin{align}
(I_m) \;\;\;\; \min_z \; & \sum_{a \in A_p}\sum_{i \in [n]} f_{a,i} z_{a,i}, \label{initial_budget_obj}\\
\text{s.t.} \; & (\ref{compact_constr_pipe_5}), \notag\\
& z_{a,i} \in \{0,1\}, \qquad a \in A_p, i \in [n] \\
& z_a \in \{0,1\}, \qquad a \in A_{cp}.
\end{align}
In this model, the objective function (\ref{initial_budget_obj}) is the same as the design model which minimizes the construction cost of the pipes. Constraint (\ref{compact_constr_pipe_5}) allows exactly one diameter choice for each pipe to be selected. In ($I_m$), we only consider the selection of diameter choices and configurations of the compressors. This integer program aims to obtain the cheapest construction cost of the pipes along with the compressor configurations, and can be solved very quickly due to the much smaller feasible space and simpler structure compared to ($P_m$). We can similarly compute the potential loss coefficients and flow limits based on the optimal diameter choices as shown in (\ref{expansion_pressure_loss}) and (\ref{expansion_flow_limits}). 

The subproblem is a variant of the nomination validation problem which includes the configurations of valves and control valves to check if the diameter choices and compressor configurations are feasible with respect to the nomination. The subproblem, denoted by ($I_s$), is the same as ($P_s$), except for the constraints on the control valves and the valves since we do not obtain configurations for them from the master problem ($I_m$) in contrast to ($P_m$). We list the changes from ($P_s$) to obtain ($I_s$).

\textbf{Variables:} we add the binary variables for the on and off states of control valves and valves. 

\textbf{Constraints:} \texttt{Comp\_and\_cont\_valve(\_nl):} we fix the configurations of compressors obtained in ($I_m$) and consequently the constraints for compressors are now linear and free of $M$. For the control valves, our preliminary studies suggest the use of \texttt{comp\_and\_cont\_valve\_nl} block.

\texttt{Valve(\_nl):} Our preliminary studies suggest the use of \texttt{Valve\_nl} block for valves.

There are no changes to other constraints.

Solving the subproblem ($I_s$) has two outcomes. If the cheapest diameter choices and compressor configurations are infeasible with respect to the nomination, we add an integer no-good cut for that set of diameter choices and compressor configurations similar to (\ref{subproblem_nogoodcut}) to the master problem ($I_m$) and resolve. Otherwise, we obtain the optimal budget, i.e., the optimal budget for this nomination is the corresponding objective value of ($I_m$). The initial budget search procedure can be run for a certain time or a fixed number of iterations. It produces an objective value below which there is no feasible budget, thus producing an initial dual bound. This lower bound on the optimal objective function value of the network design problem  can be used to initialize the binary search. 

\section{Numerical experiments} \label{sec:numerical_experiments}
\subsection{Instances} 
Our numerical experiments are based on the GasLib library networks. In a preliminary study, we concluded that the design problem based on the largest GasLib-4197 network remains challenging to solve with our framework. As a result, we consider GasLib-11, GasLib-24, GasLib-40, GasLib-134, and GasLib-582 networks. Note that GasLib-11, GasLib-24, and GasLib-40 are considered to be simple test networks while GasLib-134 and GasLib-582 are considered to be the realistic networks by~\cite{schmidt2015}. The size of the networks is given in Tables~\ref{instance_nodes} and~\ref{instance_arcs}. Depending on the nomination, a source node may have zero supply and a sink node may have zero demand. 

\begin{table}[!htbp]
\centering
\begin{tabular}{|c|c|c|c|c|}
\hline 
Name & Nodes & Sources & Sinks & In-nodes \\\hline\hline
GasLib-11 & 11 & 3 & 3 & 5 \\\hline
GasLib-24 & 24 & 3 & 5 & 16 \\\hline
GasLib-40 & 40 & 3 & 29 & 8 \\\hline
GasLib-134 & 134 & 3 & 45 & 86 \\\hline
GasLib-582 & 582 & 31 & 129 & 422\\\hline
\end{tabular}
\caption{Nodes in GasLib networks} \label{instance_nodes}
\end{table}

\begin{table}[!htbp]
\centering
\begin{tabular}{|c|c|c|c|c|c|c|}
\hline 
Name & Pipes & Short pipes & Resistors & Compressors & Control valves & Valves \\\hline\hline
GasLib-11 & 8 & 0 & 0 & 2 & 0 & 1 \\\hline
GasLib-24 & 19 & 1 & 1 & 3 & 1 & 0 \\\hline
GasLib-40 & 39 & 0 & 0 & 6 & 0 & 0 \\\hline
GasLib-134 & 86 & 45 & 0 & 1 & 1 & 0 \\\hline
GasLib-582 & 278 & 269 & 8 & 5 & 23 & 26\\\hline
\end{tabular}
\caption{Arcs in GasLib networks} \label{instance_arcs}
\end{table}

For the smaller networks of GasLib-11, GasLib-24, and GasLib-40, there is one nomination given along with the network. For the larger networks of GasLib-134 and GasLib-582, there are numerous nominations given. For GasLib-134 network, the nominations are named after the day on which the nomination is based and we pick two nominations. The nominations given with the GasLib-582 network are divided into five categories, namely, warm, mild, cool, cold, and freezing, to simulate the temperature conditions. There are two observations about the nominations. Firstly, as temperature conditions change from warm to freezing, the nominations become more demanding. There are more sinks with positive demands and the magnitudes of demands increase. Secondly, the nominations from the same temperature category vary much less compared to nominations across temperature categories. Therefore, we pick one nomination from each temperature category for the experiments. For all networks, we vary the nomination by stress levels similar to~\cite{borraz-Sanchez2015}. In particular, we use the stress levels $\{0,1,\; 0.5, \; 1.0,\; 1.5,\; 2.0\}$ and multiply the demand $d_v$ for each $v \in \mathcal{V}$ in a nomination by each stress level to create an instance. For each pipe, based on the diameter given in the GasLib networks, we use multipliers from the set $\{0.8,\; 1.0,\; 1.3,\; 1.5\}$ to create 4 different diameter choices. 

\subsection{Implementation considerations and settings} \label{implementation_considerations}
We first provide some notes on the implementation of the primal bound loop. As the objective function (\ref{primal_master_obj}) in the master problem ($P_m$) is cubic in the flow variables, there are several possible ways to implement it: 
\begin{itemize}
	\item There are nonlinear mixed integer program solvers that can take the master problem ($P_m$) as it is, for example, BARON~\cite{ts:05,ks:18} and SCIP~\cite{Achterberg2009,GamrathEtal2020OO}.
	\item The cubic objective function is second-order cone representable. For each of the cubic terms in the flow variable, we can introduce an additional variable. Consequently, we obtain a constraint in the form of $q^3 \le t$ where $q$ represents the flow variable and $t$ represents the new variable that is an upper bound to $q^3$ and we can write the second-order cone representation of $q^3 \le t$ by
	\begin{equation}
	s \ge 0, \; s + q \ge 0, \; (s + q)^2 \le w,\; w^2 \le t(s+q).  	
	\end{equation}
	The resulting second-order cone program can be handled by specialized solvers such as MOSEK~\cite{mosek}. 
	\item To take advantage of the Gurobi's~\cite{gurobi} improved capability in solving quadratic programs, for each of the cubic terms in the flow variable $q$ in the objective function, we introduce an additional variable $q_{qua}$ with $q_{qua} = q^2$. Consequently, we have a bilinear term $q q_{qua}$ in the objective function with an additional constraint. The constraint $q_{qua} = q^2$ can be re-written into the convex constraint $q_{qua} \ge q^2$. Moreover, for the pipes, as we have binary variables corresponding to the diameter choices, the convex constraint $q_{qua} \ge q^2$ can be strengthened to $q_{qua} z \ge q^2$ by perspective strengthening (see~\cite{frangioni2006}) where $z$ represents the binary variable for the diameter choice. 
\end{itemize}

We implemented all three methods. Even though the original cubic formulation (used with BARON 22.9.1~\cite{sahinidis:baron:22.9.1} and SCIP 7.0.1~\cite{GamrathEtal2020OO}) and the second-order cone formulation (used with MOSEK 9.3.10~\cite{mosek}) are convex, these solvers tend to be slower due to the presence of the binary variables. On the other hand, Gurobi 9.5.1~\cite{gurobi} is able to handle the reformulation of the cubic objective function well. As a result, we decided to use Gurobi 9.5.1 to solve ($P_m$) and ($P_s$). This also gives us the opportunity to study the impact of perspective strengthening on the computational speed.

In addition, since we only use the values of the binary variables of the pipe diameter choices and active system component configurations from the master problem ($P_m$) to fix the corresponding variables in the subproblem ($P_s$), we do not need to solve the master problem ($P_m$) to optimality. We can either set a time limit or a non-default optimality gap and we opt to use a time limit of $60$ seconds. 

We run the experiments on a computer with an Intel i7 CPU (4.20GHz) with 16GB RAM. The computer runs the Ubuntu 20.04 LTS operating system. The framework is coded in Python with Pyomo. We use Gurobi 9.5.1 to solve problems ($I_m$) and ($I_s$) as well. Algorithm~\ref{overall_procedure} shows the exact steps we use to solve the problem combining the procedures from the primal bound loop, the binary search on budget, and the initial budget search.

\begin{algorithm}[!htbp]
\caption{Overall procedure} \label{overall_procedure}
Initial budget search ($I_m$) and ($I_s$) is run for 10 min. \Comment*[h]{Initial budget search phase}\\
\lIf{a feasible budget is obtained}{terminate with the optimal budget for this nomination}
\Else(\Comment*[h]{Binary search phase} \label{binary_search_overall_procedure}){Set starting budget based on the returned value from initial budget search;\\
Binary search is run for 5 hr; for each candidate budget, primal bound loop ($P_m$) and ($P_s$) is run for 45 min to check if the budget is feasible.\\
}
\Return{$\overline{C}$ and $\underline{C}$ from binary search} 
\end{algorithm}
A note on the time limit for the initial budget search. We performed studies to extend the time limit to longer than 10 min and the improvement on the returned value is not significant. As a result, we decided to limit the initial budget search to 10 minutes.  

\subsection{Results}
In this section, we present the computational results. In addition to the results from our proposed framework, we provide some discussions about the MINLP formulations and another approach adapted from~\cite{borraz-Sanchez2015} using a computational study on nomination warm\_31 from GasLib-582, which comes from the least demanding temperature categories. The papers by~\cite{borraz-Sanchez2015} and~\cite{shiono2016} are two recent works on gas network expansion. The work of~\cite{shiono2016} specifically considers a tree-like network while the GasLib-582 network contains cycles. Although~\cite{borraz-Sanchez2015} mainly focuses on the reinforcement problem, the approach can be modified to tackle the design problem.

Recall that the MINLP formulations comprise an objective function (\ref{compact_obj}) along with constraint blocks \texttt{Flow\_conserv}, \texttt{Bound}, \texttt{Pipe}, \texttt{Short\_pipe}, \texttt{Resistor}, \texttt{Comp\_and\_cont\_valve} or \texttt{Comp\_and\_cont\_valve\_nl}, and \texttt{Valve} or \texttt{Valve\_nl}. In particular, we have two different MINLP formulations, one utilizing \texttt{Comp\_and\_cont\_valve} and \texttt{Valve} blocks while the other utilizing \texttt{Comp\_and\_cont\_valve\_nl} and \texttt{Valve\_nl} blocks. Solvers that can solve nonconvex MINLPs are considered. Note that CPLEX currently only supports nonconvexity in the objective function and hence is not applicable. We also discovered that Gurobi at times returns solutions with large constraint violations for this MINLP formulation. As a result, we focused on SCIP and BARON and we report the results on solving the MINLP formulation with BARON, which performed better in our tests.

\subsubsection{Small networks}
We consider the small networks of GasLib-11, GasLib-24, and GasLib-40 in this section. Based on the upper bound $\overline{C}$ and lower bound $\underline{C}$ from the procedure described in Algorithm~\ref{overall_procedure}, we define the percentage gap by
\begin{equation}
{\rm gap} = \frac{\overline{C} - \underline{C}}{\underline{C}} \times 100\%. \label{percentage_gap}
\end{equation}
We present the detailed results in Tables~\ref{results_gaslib_11_1} to~\ref{results_gaslib_40_2}. In particular, Tables~\ref{results_gaslib_11_1} to~\ref{results_gaslib_40_1} report the bounds $\overline{C}$ and $\underline{C}$ in $10^9$, percentage gap, and time information from the framework with and without perspective strengthening and better result from solving the two MINLP formulations with BARON as a comparison. We also show the time it takes to reach 20\% optimality gap ``20\%'' in a bracket after the time for instances with optimality gaps larger than zero, but less than 20\% at termination. The column ``Imp'' reports the improvement in the gaps from the perspective strengthening formulation. Additionally, we make the better upper bound $\overline{C}$ between our framework and BARON bold. The iteration information is reported in Tables~\ref{results_gaslib_11_2} to~\ref{results_gaslib_40_2}. The column ``Initial budget'' indicates whether the nomination is solved by the initial budget search procedure outlined in Section~\ref{budget_initialization}. The columns ``Binary search'' and ``Primal bound'' report the number of budget checked and the number of no-good cuts (\ref{subproblem_nogoodcut}) added.

For the smaller networks, we see the performance of the proposed framework is comparable to BARON up to stress level of $1.0$. For the larger stress levels, the framework returns larger gaps. It is challenging to prove an infeasible budget because of the large combinatorial space created from the diameter choices and the configurations of the active system components. Nonetheless, as the size of the network and stress level increase in GasLib-40, we see that BARON cannot close the gap to obtain the optimal solution. In addition, the perspective strengthening formulation provides improvements of about $40\%$ on average in few instances while the number of budget checked remains the same in most instances. However, the master problem ($P_m$) with the perspective formulation tends to be more difficult to solve as fewer numbers of no-good cuts are added for the same stress level in all three networks.

\begin{sidewaystable}
\centering
{\footnotesize
\begin{tabular}{|c|c|c|c|c|c|c|c|c|c|c|c|c|c|c|c|c|}
\hline
\multirow{2}{*}{Stress} & \multicolumn{4}{c|}{Without perspective strengthening} & \multicolumn{4}{c|}{With perspective strengthening} & \multirow{2}{*}{Imp(\%)} & \multicolumn{4}{c|}{BARON}\\\cline{2-9}\cline{11-14}
& $\overline{C}$ & $\underline{C}$ & Gap & Time (sec.) & $\overline{C}$ & $\underline{C}$ & Gap & Time (sec.) & & $\overline{C}$ & $\underline{C}$ & Gap & Time (sec.) \\\hline
0.1 & \textbf{1.108} & 1.108 & 0.00 & 0.40 & \textbf{1.108} & 1.108 & 0.00 & 0.40 & - & \textbf{1.108} & 1.108 & 0.00 & 0.11 \\\hline
0.5 & \textbf{1.108} & 1.108 & 0.00 & 0.30 & \textbf{1.108} & 1.108 & 0.00 & 0.30 & - & \textbf{1.108} & 1.108 & 0.00 & 0.055 \\\hline
1.0 & \textbf{1.108} & 1.108 & 0.00 & 0.30 & \textbf{1.108} & 1.108 & 0.00 & 0.30 & - & \textbf{1.108} & 1.108 & 0.00 & 0.39 \\\hline
1.5 & 1.792 & 1.417 & 26.46 & - & 1.534 & 1.417 & 8.26 & 77.74 (20\%) & 68.8 & \textbf{1.521} & 1.521 & 0.00 & 1.03 \\\hline
2.0 & 2.127 & 1.417 & 56.46 & - & 1.961 & 1.417 & 38.4 & - & 31.99 & \textbf{1.952} & 1.952 & 0.00 & 1.38 \\\hline
\end{tabular}
\caption{Computational results for GasLib-11} \label{results_gaslib_11_1}

\begin{tabular}{|c|c|c|c|c|c|c|c|c|c|c|c|c|c|c|c|c|}
\hline
\multirow{2}{*}{Stress} & \multicolumn{4}{c|}{Without perspective strengthening} & \multicolumn{4}{c|}{With perspective strengthening} & \multirow{2}{*}{Imp(\%)} & \multicolumn{4}{c|}{BARON}\\\cline{2-9}\cline{11-14}
& $\overline{C}$ & $\underline{C}$ & Gap & Time (sec.) & $\overline{C}$ & $\underline{C}$ & Gap & Time (sec.) & & $\overline{C}$ & $\underline{C}$ & Gap & Time (sec.) \\\hline
0.1 & \textbf{10.28} & 10.28 & 0.00 & 0.69 & \textbf{10.28} & 10.28 & 0.00 & 0.69 & - & \textbf{10.28} & 10.28 & 0.00 & 0.33 \\\hline
0.5 & \textbf{10.28} & 10.28 & 0.00 & 0.56 & \textbf{10.28} & 10.28 & 0.00 & 0.56 & - & \textbf{10.28} & 10.28 & 0.00 & 0.41 \\\hline
1.0 & \textbf{10.28} & 10.28 & 0.00 & 0.55 & \textbf{10.28} & 10.28 & 0.00 & 0.56 & - & \textbf{10.28} & 10.28 & 0.00 & 0.48 \\\hline
1.5 & \textbf{10.34} & 10.34 & 0.00 & 3.50 & \textbf{10.34} & 10.34 & 0.00 & 3.47 & - & \textbf{10.34} & 10.34 & 0.00 & 0.52 \\\hline
2.0 & 12.69 & 10.53 & 20.51 & - & 12.26 & 10.53 & 16.43 & 5091.11 (20\%) & 19.89 & \textbf{11.36} & 11.36 & 0.00 & 13.22 \\\hline
\end{tabular}
\caption{Computational results for GasLib-24} \label{results_gaslib_24_1}

\begin{tabular}{|c|c|c|c|c|c|c|c|c|c|c|c|c|c|c|c|c|}
\hline
\multirow{2}{*}{Stress} & \multicolumn{4}{c|}{Without perspective strengthening} & \multicolumn{4}{c|}{With perspective strengthening} & \multirow{2}{*}{Imp(\%)} & \multicolumn{4}{c|}{BARON}\\\cline{2-9}\cline{11-14}
& $\overline{C}$ & $\underline{C}$ & Gap & Time (sec.) & $\overline{C}$ & $\underline{C}$ & Gap & Time (sec.) & & $\overline{C}$ & $\underline{C}$ & Gap & Time (sec.) \\\hline
0.1 & \textbf{8.43} & 8.43 & 0.00 & 6.50 & \textbf{8.43} & 8.43 & 0.00 & 6.50 & - & \textbf{8.43} & 8.43 & 0.00 & 17.68 \\\hline
0.5 & \textbf{8.43} & 8.43 & 0.00 & 6.50 & \textbf{8.43} & 8.43 & 0.00 & 6.50 & - & \textbf{8.43} & 8.43 & 0.00 & 19.44 \\\hline
1.0 & \textbf{9.56} & 8.45 & 13.14 & 2955.35 (20\%) & \textbf{9.56} & 8.45 & 13.14 & 3019.47 (20\%) & - & \textbf{9.56} & 9.56 & 0.00 & 2651.72 \\\hline
1.5 & \textbf{11.22} & 8.45 & 32.78 & - & \textbf{11.22} & 10.34 & 32.78 & - & - & 11.62 & 10.68 & 8.80 & 17338.74 (20\%) \\\hline
2.0 & \textbf{15.38} & 8.45 & 82.01 & - & \textbf{15.38} & 10.53 & 82.01 & - & - & 17.41 & 11.69 & 48.93 & - \\\hline
\end{tabular}
\caption{Computational results for GasLib-40} \label{results_gaslib_40_1}

}
\end{sidewaystable}

\begin{table}[!htbp]
\centering
\begin{tabular}{|c|c|c|c|c|c|c|}
\hline
\multirow{2}{*}{Stress} & \multicolumn{3}{c|}{Without perspective strengthening} & \multicolumn{3}{c|}{With perspective strengthening} \\\cline{2-7}
& \begin{tabular}{@{}c@{}}Initial \\ budget \end{tabular} & \begin{tabular}{@{}c@{}}Binary \\ search \end{tabular}& \begin{tabular}{@{}c@{}} Primal \\ bound \end{tabular} & \begin{tabular}{@{}c@{}}Initial \\ budget \end{tabular} & \begin{tabular}{@{}c@{}}Binary \\ search \end{tabular}& \begin{tabular}{@{}c@{}} Primal \\ bound \end{tabular} \\\hline
0.1 & Yes & - & - & Yes & - & - \\\hline
0.5 & Yes & - & - & Yes & - & - \\\hline
1.0 & Yes & - & - & Yes & - & - \\\hline
1.5 & No & 12 & 37423 & No & 12 & 31080 \\\hline
2.0 & No & 9 & 36156 & No & 10 & 20257 \\\hline
\end{tabular}
\caption{Iteration information for GasLib-11} \label{results_gaslib_11_2}
\end{table}

\begin{table}[!htbp]
\centering
\begin{tabular}{|c|c|c|c|c|c|c|c|c|c|c|c|c|c|c|}
\hline
\multirow{2}{*}{Stress} & \multicolumn{3}{c|}{Without perspective strengthening} & \multicolumn{3}{c|}{With perspective strengthening} \\\cline{2-6}
& \begin{tabular}{@{}c@{}}Initial \\ budget search\end{tabular}& \begin{tabular}{@{}c@{}}Binary \\ search \end{tabular}& \begin{tabular}{@{}c@{}} Primal \\ bound \end{tabular} & \begin{tabular}{@{}c@{}}Initial \\ budget search\end{tabular}& \begin{tabular}{@{}c@{}}Binary \\ search \end{tabular}& \begin{tabular}{@{}c@{}} Primal \\ bound \end{tabular} \\\hline
0.1 & Yes & - & - & Yes & - & -\\\hline
0.5 & Yes & - & - & Yes & - & -\\\hline
1.0 & Yes & - & - & Yes & - & -\\\hline
1.5 & Yes & - & - & Yes & - & -\\\hline
2.0 & No & 11 & 17873 & No & 11 & 10239 \\\hline
\end{tabular}
\caption{Iteration information for GasLib-24} \label{results_gaslib_24_2}
\end{table}

\begin{table}[!htbp]
\centering
\begin{tabular}{|c|c|c|c|c|c|c|c|c|c|c|c|c|c|c|}
\hline
\multirow{2}{*}{Stress} & \multicolumn{3}{c|}{Without perspective strengthening} & \multicolumn{3}{c|}{With perspective strengthening} \\\cline{2-6}
& \begin{tabular}{@{}c@{}}Initial \\ budget search\end{tabular}& \begin{tabular}{@{}c@{}}Binary \\ search \end{tabular}& \begin{tabular}{@{}c@{}} Primal \\ bound \end{tabular} & \begin{tabular}{@{}c@{}}Initial \\ budget search\end{tabular}& \begin{tabular}{@{}c@{}}Binary \\ search \end{tabular}& \begin{tabular}{@{}c@{}} Primal \\ bound \end{tabular} \\\hline
0.1 & Yes & - & - & Yes & - & -\\\hline
0.5 & Yes & - & - & Yes & - & -\\\hline
1.0 & No & 13 & 596 & No & 13 & 326\\\hline
1.5 & No & 11 & 494 & Yes & 11 & 322\\\hline
2.0 & No & 10 & 302 & No & 10 & 321 \\\hline
\end{tabular}
\caption{Iteration information for GasLib-40} \label{results_gaslib_40_2}
\end{table}

\subsubsection{GasLib-134}
For the GasLib-134 network, we pick nominations 2011-11-27 and 2016-01-11. The results are reported in identical format as the smaller networks in Tables~\ref{results_gaslib_134_n1_1} to~\ref{results_gaslib_134_n2_2} and we use the same definition for percentage gap from (\ref{percentage_gap}). For both nominations, we see that with smaller stress levels, the problems can be directly solved by the initial budget search procedure which outperforms BARON. As the stress levels increase, although BARON returns solutions with optimality gaps, it takes shorter time for BARON to reach 20\% optimality gaps. For this network, however, the perspective strengthening formulation does not show any improvements except in only one instance. Note that the networks are generally independent and although the size of the network is a significant factor in the difficulty of the instances, there are other factors. We observe larger number of budgets checked and no-good cuts added than those in a smaller GasLib-40, but the trend of number of budget checked and no-good cuts across different stress levels is the same as that in the smaller networks.

\begin{sidewaystable}
{\footnotesize
\centering
\begin{tabular}{|c|c|c|c|c|c|c|c|c|c|c|c|c|c|c|c|c|}
\hline
\multirow{2}{*}{Stress} & \multicolumn{4}{c|}{Without perspective strengthening} & \multicolumn{4}{c|}{With perspective strengthening} & \multirow{2}{*}{Imp(\%)} & \multicolumn{4}{c|}{BARON}\\\cline{2-9}\cline{11-14}
& $\overline{C}$ & $\underline{C}$ & Gap & Time (sec.) & $\overline{C}$ & $\underline{C}$ & Gap & Time (sec.) & & $\overline{C}$ & $\underline{C}$ & Gap & Time (sec.) \\\hline
0.1 & \textbf{8.23} & 8.23 & 0.00 & 0.32 & \textbf{8.23} & 8.23 & 0.00 & 0.32 & - & \textbf{8.23} & 8.23 & 0.00 & 7.78 \\\hline
0.5 & \textbf{8.23} & 8.23 & 0.00 & 0.29 & \textbf{8.23} & 8.23 & 0.00 & 0.23 & - & \textbf{8.23} & 8.23 & 0.00 & 7.12 \\\hline
1.0 & \textbf{8.34} & 8.23 & 1.34 & 64.0 (20\%) & \textbf{8.34} & 8.23 & 1.34 & 72.15 (20\%) & - & \textbf{8.34} & 8.34 & 0.00 & 43.54 \\\hline
1.5 & \textbf{9.77} & 8.23 & 18.71 & 2887.02 (20\%) & \textbf{9.77} & 8.23 & 18.71 & 2778.8 (20\%) & - & \textbf{9.77} & 9.77 & 0.00 & 4621.11 \\\hline
2.0 & 12.57 & 8.23 & 52.73 & - & 12.35 & 8.23 & 50.06 & - & 5.06 & \textbf{12.04} & 11.44 & 5.24 & 1291.31 (20\%) \\\hline
\end{tabular}
\caption{Computational results for GasLib-134 and nomination 2011-11-27}\label{results_gaslib_134_n1_1}

\begin{tabular}{|c|c|c|c|c|c|c|c|c|c|c|c|c|c|c|c|c|}
\hline
\multirow{2}{*}{Stress} & \multicolumn{4}{c|}{Without perspective strengthening} & \multicolumn{4}{c|}{With perspective strengthening} & \multirow{2}{*}{Imp(\%)} & \multicolumn{4}{c|}{BARON}\\\cline{2-9}\cline{11-14}
& $\overline{C}$ & $\underline{C}$ & Gap & Time (sec.) & $\overline{C}$ & $\underline{C}$ & Gap & Time (sec.) & & $\overline{C}$ & $\underline{C}$ & Gap & Time (sec.) \\\hline
0.1 & \textbf{8.23} & 8.23 & 0.00 & 0.3 & \textbf{8.23} & 8.23 & 0.00 & 0.35 & - & \textbf{8.23} & 8.23 & 0.00 & 0.41 \\\hline
0.5 & \textbf{8.23} & 8.23 & 0.00 & 0.25 & \textbf{8.23} & 8.23 & 0.00 & 0.28 & - & \textbf{8.23} & 8.23 & 0.00 & 8.28 \\\hline
1.0 & \textbf{8.34} & 8.23 & 1.34 & 0.21 & \textbf{8.34} & 8.23 & 1.34 & 0.23 & - & \textbf{8.34} & 8.34 & 0.00 & 9.69 \\\hline
1.5 & \textbf{9.68} & 8.23 & 17.62 & 2768.39 (20\%) & \textbf{9.68} & 8.23 & 17.62 & 2775.2 (20\%) & - & 9.69 & 9.41 & 2.98 & 180.28 (20\%) \\\hline
2.0 & \textbf{12.72} & 8.23 & 54.56 & - & \textbf{12.72} & 8.23 & 54.56 & - & - & 12.86 & 12.11 & 6.19 & 3542.7 (20\%) \\\hline
\end{tabular}
}
\caption{Computational results for GasLib-134 and nomination 2016-01-11}\label{results_gaslib_134_n2_1}
\end{sidewaystable}

\begin{table}[!htbp]
\centering
\begin{tabular}{|c|c|c|c|c|c|c|c|c|c|c|c|c|c|c|}
\hline
\multirow{2}{*}{Stress} & \multicolumn{3}{c|}{Without perspective strengthening} & \multicolumn{3}{c|}{With perspective strengthening} \\\cline{2-6}
& \begin{tabular}{@{}c@{}}Initial \\ budget search\end{tabular}& \begin{tabular}{@{}c@{}}Binary \\ search \end{tabular}& \begin{tabular}{@{}c@{}} Primal \\ bound \end{tabular} & \begin{tabular}{@{}c@{}}Initial \\ budget search\end{tabular}& \begin{tabular}{@{}c@{}}Binary \\ search \end{tabular}& \begin{tabular}{@{}c@{}} Primal \\ bound \end{tabular} \\\hline
0.1 & Yes & - & - & Yes & - & -\\\hline
0.5 & Yes & - & - & Yes & - & -\\\hline
1.0 & No & 15 & 10681 & No & 15 & 10152\\\hline
1.5 & No & 11 & 9857 & Yes & 11 & 4325\\\hline
2.0 & No & 11 & 8760 & No & 9 & 3511 \\\hline
\end{tabular}
\caption{Iteration information for GasLib-134 and nomination 2011-11-27} \label{results_gaslib_134_n1_2}
\end{table}

\begin{table}[!htbp]
\centering
\begin{tabular}{|c|c|c|c|c|c|c|c|c|c|c|c|c|c|c|}
\hline
\multirow{2}{*}{Stress} & \multicolumn{3}{c|}{Without perspective strengthening} & \multicolumn{3}{c|}{With perspective strengthening} \\\cline{2-6}
& \begin{tabular}{@{}c@{}}Initial \\ budget search\end{tabular}& \begin{tabular}{@{}c@{}}Binary \\ search \end{tabular}& \begin{tabular}{@{}c@{}} Primal \\ bound \end{tabular} & \begin{tabular}{@{}c@{}}Initial \\ budget search\end{tabular}& \begin{tabular}{@{}c@{}}Binary \\ search \end{tabular}& \begin{tabular}{@{}c@{}} Primal \\ bound \end{tabular} \\\hline
0.1 & Yes & - & - & Yes & - & -\\\hline
0.5 & Yes & - & - & Yes & - & -\\\hline
1.0 & Yes & - & - & Yes & - & -\\\hline
1.5 & No & 12 & 10046 & Yes & 12 & 4358\\\hline
2.0 & No & 11 & 8962 & No & 11 & 3512 \\\hline
\end{tabular}
\caption{Iteration information for GasLib-134 and nomination 2016-01-11} \label{results_gaslib_134_n2_2}
\end{table}

\subsubsection{GasLib-582}
For GasLib-582 network, we pick nominations warm\_31, mild\_3838, cool\_2803, cold\_4218, and freezing\_188. In this set of experiments, we discover that BARON is able to obtain a lower bound and constantly improve it, but it tends to be slower in finding a feasible solution to close the gap after hours of computation. On the other hand, our framework is very efficient in finding feasible solutions. As a result, we decide to run BARON for 45 minutes before the binary search phase, i.e., before step~\ref{binary_search_overall_procedure} in Algorithm~\ref{overall_procedure} and we report the better lower bound $\underline{C}$ between our framework and BARON. 

We first briefly recap the approach proposed in~\cite{borraz-Sanchez2015} where the authors construct a convex mixed-integer second-order cone (MISOC) relaxation of the reinforcement formulation, and fix all the binary decision variables after solving the relaxation to obtain a nomination validation problem. The resulting nonlinear program is then solved by a solver to obtain a solution if the nonlinear program is feasible. If the nonlinear program is infeasible, then the authors propose to use any feasible solutions to the relaxation. We adapt a similar MISOC relaxation for the pipes and resistors. We leave the details of this relaxation in Appendix~\ref{sec:appendix_b}. In our computations, following the procedure and solving the nomination validation problem did not yield feasible solutions for nomination warm\_31 with any stress levels. 

Next, we present the results from our framework in Tables~\ref{results_gaslib_582_warm_31} to~\ref{results_gaslib_582_freezing_188}. We combine the computational results and the iteration information into one table in this section. In the ``Iteration'' column, the first number is the number of budget checked and the second number is the number of no-good cuts added. Everything else follows identical format as discussed previously. Additionally, Table~\ref{gap_results} summarizes the gaps across different nominations and stress levels.

From the results, we see that our framework is able to find a feasible budget for all 25 instances. In particular, it provides an optimal budget for twelve instances and a budget with less than $20\%$ gaps for another six instances.  There are a few instances where we reached the time limit with large gaps. We mark these instances in bold. These instances are with higher stress levels and/or worse temperature conditions. As we increase the stress level and/or deteriorate the temperature conditions (from warm to freezing) making the nominations more demanding, we observe that it becomes more difficult to find feasible solutions in the primal bound loops to prove a feasible budget and thus close the gap by binary search. The primal bound loops hit the time limit much more often. In addition, for all instances, we are not able to prove infeasible budget from the primal bound loop. While a large number of binary solutions are feasible to the master problem ($P_m$), each integer no-good cut only invalidates one of them. As a result, the lower bounds on budget, $\underline{C}$, are almost the same across different nominations and stress levels. 

Furthermore, the perspective strengthening is shown to be effective in closing the gaps for higher stress levels. There is no instance for which the implementation without perspective strengthening achieves better gaps than the implementation with perspective strengthening. The average and largest improvements from perspective strengthening are about $55\%$ (excluding the instances that are solved to optimality both with and without perspective strengthening) and $97\%$, respectively. The improvements are all due to obtaining better feasible solutions. 

Similar to the results on other GasLib networks, we see that, with more demanding nominations and larger stress levels, fewer budget are checked and no-good cuts added. However, for instances with larger gaps (for example, those that are marked in bold), many diameter choices and configurations of the active system components can be evaluated relatively fast to be infeasible in the subproblem ($P_s$) resulting in large number of no-good cuts added for every budget checked. Overall, the initial budget search procedure is effective as some of these challenging instances are solved with the procedure while a lower bound is provided to other instances to assess the quality of the feasible solutions obtained by the primal bound loop.

\begin{sidewaystable}
{\footnotesize
\centering
\begin{tabular}{|c|c|c|c|c|c|c|c|c|c|c|c|c|c|c|c|c|c|c|}
\hline
\multirow{2}{*}{Stress} & \multicolumn{6}{c|}{Without perspective strengthening} & \multicolumn{6}{c|}{With perspective strengthening} & \multirow{2}{*}{Imp(\%)} \\\cline{2-13}
& $\overline{C}$ & $\underline{C}$ & Gap & Time (sec.) & \begin{tabular}{@{}c@{}}Initial \\ budget \end{tabular} & Iteration & $\overline{C}$ & $\underline{C}$ & Gap & Time (sec.) & \begin{tabular}{@{}c@{}}Initial \\ budget \end{tabular}& Iteration \\\hline
0.1 & 12.59 & 12.59 & 0.00 & 1.39 & Yes & -(-) & 12.59 & 12.59 & 0.00 & 1.41 & Yes & -(-) & -\\\hline
0.5 & 12.59 & 12.59 & 0.00 & 1.22 & Yes & -(-) & 12.59 & 12.59 & 0.00 & 1.11 & Yes & -(-) & -\\\hline
1.0 & 12.87 & 12.59 & 2.22 & 2102.7 (20\%) & No & 12(303) & 12.59 & 12.59 & 0.06 & 1397.14 (20\%) & No & 18(615) & 97\\\hline
1.5 & 14.65 & 12.77 & 14.72 & 7589.99 (20\%) & No & 10(306) & 13.77 & 12.77 & 7.83 & 6222.75 (20\%) & No & 9(255) & 47\\\hline
2.0 & 19.22 & 13.40 & 43.43 & - & No & 9(458) & 17.40 & 13.40 & 29.85 & - & No & 10(287) & 31\\\hline
\end{tabular}
\caption{Computational results for GasLib-582 and nomination warm\_31}\label{results_gaslib_582_warm_31}

\begin{tabular}{|c|c|c|c|c|c|c|c|c|c|c|c|c|c|c|c|c|c|c|}
\hline
\multirow{2}{*}{Stress} & \multicolumn{6}{c|}{Without perspective strengthening} & \multicolumn{6}{c|}{With perspective strengthening} & \multirow{2}{*}{Imp(\%)} \\\cline{2-13}
& $\overline{C}$ & $\underline{C}$ & Gap & Time (sec.) & \begin{tabular}{@{}c@{}}Initial \\ budget \end{tabular} & Iteration & $\overline{C}$ & $\underline{C}$ & Gap & Time (sec.) & \begin{tabular}{@{}c@{}}Initial \\ budget \end{tabular}& Iteration \\\hline
0.1 & 12.59 & 12.59 & 0.00 & 1.69 & Yes & -(-) & 12.59 & 12.59 & 0.00 & 1.68 & Yes & -(-) & -\\\hline
0.5 & 12.59 & 12.59 & 0.00 & 1.68 & Yes & -(-) & 12.59 & 12.59 & 0.00 & 1.68 & Yes & -(-) & -\\\hline
1.0 & 12.59 & 12.59 & 0.00 & 1.59 & Yes & -(-) & 12.59 & 12.59 & 0.00 & 1.48 & Yes & -(-) & -\\\hline
1.5 & 18.78 & 13.18 & 42.49 & - & No & 9(376) & 15.54 & 13.18 & 17.91 & 2113.41(20\%) & No & 8(268) & 58 \\\hline
2.0 & 28.32 & 13.51 & \textbf{109.62} & - & No & 9(554) & 21.43 & 13.51 & 58.62 & - & No & 8(269) & 47 \\\hline
\end{tabular}
\caption{Computational results for GasLib-582 and nomination mild\_3838}\label{results_gaslib_582_mild_3838}

\begin{tabular}{|c|c|c|c|c|c|c|c|c|c|c|c|c|c|c|c|c|c|c|}
\hline
\multirow{2}{*}{Stress} & \multicolumn{6}{c|}{Without perspective strengthening} & \multicolumn{6}{c|}{With perspective strengthening} & \multirow{2}{*}{Imp(\%)} \\\cline{2-13}
& $\overline{C}$ & $\underline{C}$ & Gap & Time (sec.) & \begin{tabular}{@{}c@{}}Initial \\ budget \end{tabular} & Iteration & $\overline{C}$ & $\underline{C}$ & Gap & Time (sec.) & \begin{tabular}{@{}c@{}}Initial \\ budget \end{tabular}& Iteration \\\hline
0.1 & 12.59 & 12.59 & 0.00 & 1.39 & Yes & -(-) & 12.59 & 12.59 & 0.00 & 1.34 & Yes & -(-) & -\\\hline
0.5 & 12.59 & 12.59 & 0.00 & 3.11 & Yes & -(-) & 12.59 & 12.59 & 0.00 & 2.82 & Yes & -(-) & -\\\hline
1.0 & 12.65 & 12.59 & 4.77 & 3875.47 (20\%) & No & 14(371) & 12.62 & 12.59 & 2.38 & 3977.63 (20\%) & No & 15(461) & 50\\\hline
1.5 & 14.65 & 12.74 & 14.99 & 8235.91 (20\%) & No & 9(297) & 13.16 & 12.74 & 3.29 & 7122.24 (20\%) & No & 11(271) & 78 \\\hline
2.0 & 38.54 & 13.61 & \textbf{183.17} & - & No & 8(500) & 17.31 & 13.61 & 27.19 & - & No & 9(252) & 85 \\\hline
\end{tabular}
}
\caption{Computational results for GasLib-582 and nomination cool\_2803}\label{results_gaslib_582_cool_2803}
\end{sidewaystable}

\begin{sidewaystable}
{\footnotesize
\centering
\begin{tabular}{|c|c|c|c|c|c|c|c|c|c|c|c|c|c|c|c|c|c|c|}
\hline
\multirow{2}{*}{Stress} & \multicolumn{6}{c|}{Without perspective strengthening} & \multicolumn{6}{c|}{With perspective strengthening} & \multirow{2}{*}{Imp(\%)} \\\cline{2-13}
& $\overline{C}$ & $\underline{C}$ & Gap & Time (sec.) & \begin{tabular}{@{}c@{}}Initial \\ budget \end{tabular} & Iteration & $\overline{C}$ & $\underline{C}$ & Gap & Time (sec.) & \begin{tabular}{@{}c@{}}Initial \\ budget \end{tabular}& Iteration \\\hline
0.1 & 12.59 & 12.59 & 0.00 & 2.73 & Yes & -(-) & 12.59 & 12.59 & 0.00 & 2.62 & Yes & -(-) & -\\\hline
0.5 & 12.59 & 12.59 & 0.00 & 1.8 & Yes & -(-) & 12.59 & 12.59 & 0.00 & 1.83 & Yes & -(-) & -\\\hline
1.0 & 13.77 & 12.82 & 7.41 & 5157.15 (20\%) & No & 10(276) & 13.03 & 12.82 & 1.64 & 4642.22 (20\%) & No & 11(274) & 78 \\\hline
1.5 & 48.77 & 13.70 & \textbf{255.99} & - & No & 7(396) & 22.89 & 13.70 & \textbf{67.08} & - & No & 8(252) & 74\\\hline
2.0 & 37.76 & 16.24 & \textbf{132.51} & - & No & 9(462) & 33.04 & 16.24 & \textbf{103.45} & - & No & 9(276) & 22\\\hline
\end{tabular}
\caption{Computational results for GasLib-582 and nomination cold\_4128}\label{results_gaslib_582_cold_4218}

\begin{tabular}{|c|c|c|c|c|c|c|c|c|c|c|c|c|c|c|c|c|c|c|}
\hline
\multirow{2}{*}{Stress} & \multicolumn{6}{c|}{Without perspective strengthening} & \multicolumn{6}{c|}{With perspective strengthening} & \multirow{2}{*}{Imp(\%)} \\\cline{2-13}
& $\overline{C}$ & $\underline{C}$ & Gap & Time (sec.) & \begin{tabular}{@{}c@{}}Initial \\ budget \end{tabular} & Iteration & $\overline{C}$ & $\underline{C}$ & Gap & Time (sec.) & \begin{tabular}{@{}c@{}}Initial \\ budget \end{tabular}& Iteration \\\hline
0.1 & 12.59 & 12.59 & 0.00 & 1.27 & Yes & -(-) & 12.59 & 12.59 & 0.00 & 1.18 & Yes & -(-) & -\\\hline
0.5 & 12.59 & 12.59 & 0.00 & 1.13 & Yes & -(-) & 12.59 & 12.59 & 0.00 & 1.25 & Yes & -(-) & -\\\hline
1.0 & 17.23 & 14.22 & 21.16 & - & No & 9(370) & 14.95 & 14.22 & 5.13 & 9841.49 (20\%) & No & 8(249) & 76 \\\hline
1.5 & 35.39 & 14.92 & \textbf{137.20} & - & No & 8(428) & 35.39 & 14.92 & \textbf{137.20} & - & No & 8(284) & - \\\hline
2.0 & 56.63 & 17.75 & \textbf{219.04} & - & No & 9(636) & 47.19 & 17.75 & \textbf{165.86} & - & No & 8(367) & 32\\\hline
\end{tabular}
\caption{Computational results for GasLib-582 and nomination freezing\_188}\label{results_gaslib_582_freezing_188}

\begin{tabular}{|c||c|c||c|c||c|c||c|c||c|c|}
\hline	
\multirow{2}{*}{Stress} & \multicolumn{2}{c||}{warm\_31} & \multicolumn{2}{c||}{mild\_3838} & \multicolumn{2}{c||}{cool\_2803} & \multicolumn{2}{c||}{cold\_4218} & \multicolumn{2}{c|}{freezing\_188}\\\cline{2-11}
& w/o & w & w/o & w & w/o & w & w/o & w & w/o & w \\\hline\hline
0.1 & 0.00 & 0.00 & 0.00 & 0.00 & 0.00 & 0.00 & 0.00 & 0.00 & 0.00 & 0.00\\\hline
0.5 & 0.00 & 0.00 & 0.00 & 0.00 & 0.00 & 0.00 & 0.00 & 0.00 & 0.00 & 0.00\\\hline
1.0 & 2.22 & 0.06 & 0.00 & 0.00 & 4.77 & 2.38 & 7.41 & 1.64 & 21.16 & 5.13\\\hline
1.5 & 14.72 & 7.83 & 42.49 & 17.91 & 14.99 & 3.29 & \textbf{255.99} & \textbf{67.08} & \textbf{137.20} & \textbf{137.20} \\\hline
2.0 & 43.43 & 29.85 & \textbf{109.62} & 58.62 & \textbf{183.17} & 27.19 & \textbf{132.51} & \textbf{103.45} & \textbf{219.04} & \textbf{165.86}\\\hline
\end{tabular}
}
\caption{Gap in (\%) without (w/o) and with (w) perspective strengthening} \label{gap_results}
\end{sidewaystable}

\section{Conclusion} \label{sec:conclusion}

In conclusion, we studied the gas network design problem, where diameter choices of pipes and active system component configurations are decided. We proposed a decomposition framework to solve the problem. In particular, in the primal bound loop of the framework, for a given budget, we modify a convex NLP formulation to construct master problems to obtain favorable diameter choices and active system component configurations, and validate their feasibility in the subproblem. Binary search is performed as an outer loop to minimize the budget. We also proposed a procedure to obtain a good initial budget for the binary search. The proposed framework was tested on the various GasLib networks and instances were created from combining nominations under different temperature conditions and stress level multipliers. The computational results show that the framework is effective in solving most instances, especially for the large GasLib-582 network, when combined with BARON for obtaining lower bounds.

There are a few future directions that could be explored. 
The cost of operating the network may be an interest from an operator's perspective. Our framework can be adapted and applied to incorporate the cost of operations with simple modifications. 


\section*{Acknowledgement}
This work was conducted as part of the Institute for the Design of Advanced Energy Systems (IDAES) with support through the Simulation-Based Engineering, Crosscutting Research Program and the Solid Oxide Fuel Cell Program’s Integrated Energy Systems thrust within the U.S. Department of Energy’s Office of Fossil Energy and Carbon Management.

\section*{Appendix A: Proof of Theorem \ref{CVXNLP_network_eqv}} \label{sec:appendix_a}
\begin{proof}
We first consider the if part. Suppose that $(\hat{q}^+, \hat{q}^-, \hat{\lambda}, \hat{\mu}^+, \hat{\mu}^-)$ solves (CVXNLP). Consider the first-order stationary conditions for (CVXNLP) as follows:
\begin{align}
	& \phi(\hat{q}_a^+) - \hat{\mu}_a^+ - \hat{\lambda}_v + \hat{\lambda}_w = 0, \qquad a = (v,w) \in A_p, \label{thm1_opt_cond_1}\\
	& \phi(\hat{q}_a^-) - \hat{\mu}_a^- + \hat{\lambda}_v - \hat{\lambda}_w = 0, \qquad a = (v,w) \in A_p, \label{thm1_opt_cond_2}\\
	& \hat{q}_a^+, \hat{\mu}_a^+ \ge 0, \qquad a = (v,w) \in A_p, \\
	& \hat{q}_a^+ \cdot \hat{\mu}_a^+ = 0, \qquad a = (v,w) \in A_p, \\
	& \hat{q}_a^-, \hat{\mu}_a^- \ge 0, \qquad a = (v,w) \in A_p, \\
	& \hat{q}_a^- \cdot \hat{\mu}_a^- = 0, \qquad a = (v,w) \in A_p, \\
	& \sum_{a \in A_\text{in}(v)} (\hat{q}_a^+ - \hat{q}_a^-) - \sum_{a \in A_\text{out}(v)} (\hat{q}_a^+ - \hat{q}_a^-) = d_v, \; v \in \mathcal{V}.
\end{align}
First, it cannot happen that $\hat{q}_a^+, \hat{q}_a^- > 0$ for any $a \in A_p$, otherwise, we can define 
\begin{equation}
\tilde{q}_a^+ = \max\{\hat{q}_a^+ - \hat{q}_a^-, 0\}, \; \tilde{q}_a^- = \max\{0, \hat{q}_a^- - \hat{q}_a^+\},	
\end{equation}
where $\tilde{q}_a^+ \le \hat{q}_a^+$ and $\tilde{q}_a^- \le \hat{q}_a^-$. The new flow values $\tilde{q}_a^+$ and $\tilde{q}_a^-$ are feasible and because of the strict monotonicity of $\phi(\cdot)$, they result in a smaller objective value which contradicts the optimality of $\hat{q}^+$ and $\hat{q}^-$. Furthermore, the complementary slackness conditions imply that, if $\hat{q}_a^+ \; (\text{or} \; \hat{q}_a^-) > 0$, then $\hat{\mu}_a^+ \; (\text{or} \; \hat{\mu}_a^-) = 0$. If $\hat{q}_a^+ = \hat{q}_a^- = 0$ for some $a$, then adding (\ref{thm1_opt_cond_1}) and (\ref{thm1_opt_cond_2}) gives
\begin{equation}
\hat{\mu}_a^+ + \hat{\mu}_a^- = 0 \implies \hat{\mu}_a^+ = \hat{\mu}_a^- = 0.
\end{equation}

Consequently, we can simplify (\ref{thm1_opt_cond_1}) and (\ref{thm1_opt_cond_2}) by differentiating the cases on $q_a^+$ and $q_a^-$ to be 
\begin{align}
	& \phi(\hat{q}_a^+) - \hat{\lambda}_v + \hat{\lambda}_w = 0, \qquad a = (v,w) : \hat{q}_a^+ > 0,\\
	& \phi(\hat{q}_a^-)  + \hat{\lambda}_v - \hat{\lambda}_w = 0, \qquad a = (v,w): \hat{q}_a^- > 0,\\
	& \hat{\lambda}_v - \hat{\lambda}_w = 0, \qquad a = (v,w): \hat{q}_a^+ = \hat{q}_a^- = 0.
\end{align}

Now define $(\pi, q)$ as 
\begin{align}
& \pi_v = \hat{\lambda}_v, \qquad v \in \mathcal{V},\\
& q_a = \hat{q}_a^+ - \hat{q}_a^- , \qquad a \in A_p,
\end{align}
and we see $(\pi, q)$ satisfies the network analysis equations. 

Now we consider the only if part. Suppose that $(\pi, q)$ solves the network analysis equations. We define the following:
\begin{align}
& \hat{q}_a^+ = \max\{0, q_a\}, \qquad a \in A_p, \\
& \hat{q}_a^- = \vert \min\{0, q_a\} \vert, \qquad a \in A_p, \\
& \hat{\lambda}_v = \pi_v, \qquad v \in \mathcal{V}, \\
& \hat{\mu}_a^+ = \max\{0, \pi_w - \pi_v + \phi(\hat{q}_a^+)\}, \qquad a = (v, w) \in A_p, \\
& \hat{\mu}_a^- = \max\{0, \pi_v - \pi_w + \phi(\hat{q}_a^-)\}, \qquad a = (v, w) \in A_p.
\end{align}

Then $(\hat{q}^+, \hat{q}^-, \hat{\lambda}, \hat{\mu}^+, \hat{\mu}^-)$ satisfies the first-order stationary conditions. To see this, we first verify that, when $q_a \ge 0$, then $\hat{q}_a^+ = q_a \ge 0$ and $\hat{q}_a^ - = 0$. From the potential loss equation (\ref{network_ana_potential_drop}) in network analysis equations, we have that: 
	$\pi_v - \pi_w = \phi(q_a) = \phi(\hat{q}_a^+)$. Consequently, we have
\begin{align}
	\hat{\mu}_a^+ & = \max\{0, \pi_w - \pi_v + \phi(\hat{q}_a^+)\} = 0, \\
	\hat{\mu}_a^- & = \max\{0, \pi_v - \pi_w + \phi(\hat{q}_a^-)\} = \max\{0, \phi(\hat{q}_a^+) + \phi(0)\} = \phi(\hat{q}_a^+) \ge 0,
\end{align}
and 
\begin{align}
	& \phi(\hat{q}_a^+) - \hat{\mu}_a^+ - \hat{\lambda}_v + \hat{\lambda}_w = \phi(\hat{q}_a^+) - 0  - \pi_v + \pi_w = 0,\\
	& \phi(\hat{q}_a^-) - \hat{\mu}_a^- + \hat{\lambda}_v - \hat{\lambda}_w = \phi(0) - \phi(\hat{q}_a^+) + \pi_v - \pi_w = 0.
\end{align}

Similarly, we can verify for $q_a < 0$. Furthermore, the strict monotonically increasing property of $\phi(\cdot)$ implies the convexity of $\Phi(\cdot)$. The constraints in (CVXNLP) are linear and thus (CVXNLP) is convex. The satisfaction of the first-order stationary conditions is necessary and sufficient for $(\pi, q)$ to be an optimal solution to (CVXNLP) and it is the unique optimal solution due to the convexity.
\end{proof}

\section*{Appendix B: Mixed-integer second-order conic (MISOC) relaxation} \label{sec:appendix_b}
The relaxations are constructed for the pipes and resistors. For the pipes, instead of decomposing the flow variables $q_{a,i}$ into $q_{a,i}^+$ and $q_{a,i}^-$, we define two binary variables $x_{a}^+$ and $x_a^-$ for the flow directions and enforce $x_a^+ + x_a^- = 1$. If $x_a^+ = 1$, then $q_{a,i} \ge 0$ and if $x_a^- = 1$, then $q_{a,i} < 0$. In addition, we create multiple potential variables $\pi_{v,i}$ and $\pi_{w,i}$ for $a = (v,w) \in A_p$ and $i \in [n]$. Now consider a pipe $a = (v,w)$ and a diameter choice $i$, we can write the potential loss as
\begin{equation}
	(x_a^+ - x_a^-)(\pi_{v,i} - \pi_{w,i}) = \alpha_{a,i}q_{a,i}^2. \label{MISOC_1}
\end{equation}

The left-hand side of (\ref{MISOC_1}) is bilinear. If we define $\gamma_{a,i} = (x_a^+ - x_a^-)(\pi_{v,i} - \pi_{w,i})$, we can write the standard McCormick relaxation for $\gamma_{a,i} = (x_a^+ - x_a^-)(\pi_{v,i} - \pi_{w,i})$ by
\begin{align}
	& \gamma_{a,i} \ge \pi_{w,i} - \pi_{v,i} + (\pi_v^\text{min} - \pi_w^\text{max})(x_a^+ - x_a^- + 1),\\
	& \gamma_{a,i} \ge \pi_{v,i} - \pi_{w,i} + (\pi_v^\text{max} - \pi_w^\text{min})(x_a^+ - x_a^- - 1),\\
	& \gamma_{a,i} \ge \pi_{w,i} - \pi_{v,i} + (\pi_v^\text{max} - \pi_w^\text{min})(x_a^+ - x_a^- + 1),\\
	& \gamma_{a,i} \ge \pi_{v,i} - \pi_{w,i} + (\pi_v^\text{min} - \pi_w^\text{max})(x_a^+ - x_a^- - 1).
\end{align}

With $\gamma_{a,i}$ defined, constraint (\ref{MISOC_1}) can be written as $\gamma_{a,i} = \alpha_{a,i}q_{a,i}^2$ and can be further relaxed to become convex as follows: 
\begin{equation}
	\gamma_{a,i} \ge \alpha_{a,i}q_{a,i}^2.
\end{equation}

Applying perspective strengthening to the relaxed constraint gives
\begin{equation}
	z_{a,i} \gamma_{a,i} \ge \alpha_{a,i}q_{a,i}^2.
\end{equation}

Now the potential loss constraint (\ref{compact_constr_pipe_2}) for pipes becomes 
\begin{equation}
	\pi_v - \pi_w = \sum_{i \in [n]} \gamma_{a,i}. 
\end{equation}

We can create similar relaxations for the resistors. For a resistor $a = (v,w) \in A_r$, we have
\begin{align}
	& \gamma_{a} \ge \pi_{w} - \pi_{v} + (\pi_v^\text{min} - \pi_w^\text{max})(x_a^+ - x_a^- + 1),\\
	& \gamma_{a} \ge \pi_{v} - \pi_{w} + (\pi_v^\text{max} - \pi_w^\text{min})(x_a^+ - x_a^- - 1),\\
	& \gamma_{a} \ge \pi_{w} - \pi_{v} + (\pi_v^\text{max} - \pi_w^\text{min})(x_a^+ - x_a^- + 1),\\
	& \gamma_{a} \ge \pi_{v} - \pi_{w} + (\pi_v^\text{min} - \pi_w^\text{max})(x_a^+ - x_a^- - 1),\\
	& \gamma_a \ge \alpha_a q_a^2.
\end{align}

Additionally, the binary variables $x_a^\text{dir}$ in constraints (\ref{compact_constr_pipe_3})-(\ref{compact_constr_pipe_4}) and (\ref{compact_constr_resistor_2})-(\ref{compact_constr_resistor_3}) that govern flow limits on directions are replaced by $x_a^+$ and $x_a^-$ correspondingly. We keep the rest of constraints unchanged and obtain a convex MISOC relaxation of the design problem as a result.

\bibliographystyle{plain}
\bibliography{gasnetworkopt}

\end{document}